\documentclass{amsart}
\usepackage{longsalch, xy}
\xyoption{all}
\title{Relative homological algebra, Waldhausen $K$-theory, and quasi-Frobenius conditions.}
\date{March 2013}

\begin{document}

\begin{abstract}
We study the question of the existence of a Waldhausen category on any (relative) abelian category in which the contractible objects are the (relatively) projective objects.
The associated $K$-theory groups are ``stable algebraic $G$-theory,'' which in degree zero form a certain stable representation group. We prove both some existence and nonexistence results
about such Waldhausen category structures, including the fact that, while it was known that the category of $R$-modules admits a model category structure
if $R$ is quasi-Frobenius, that assumption is required even to get a Waldhausen category structure with cylinder functor---i.e., Waldhausen categories
do not offer a more general framework than model categories for studying stable representation theory of rings. We study multiplicative structures on these Waldhausen categories, and we relate stable algebraic $G$-theory to
algebraic $K$-theory and we compute stable algebraic $G$-theory for 
finite-dimensional quasi-Frobenius nilpotent extensions of finite fields.
Finally, we show that the connective stable $G$-theory spectrum of 
$\mathbb{F}_{p^n}[x]/x^{p^n}$ is a complex oriented ring spectrum, partially 
answering
a question of J. Morava about complex orientations on algebraic $K$-theory spectra.
\end{abstract}

\maketitle
\tableofcontents

\section{Introduction.}
Waldhausen's paper \cite{MR802796} defines several kinds of categorical
structure which are meaningful for algebraic $K$-theory. 
A {\em category with cofibrations and weak equivalences,} also called a 
{\em Waldhausen category}, has just enough structure for Waldhausen's machinery
to produce an associated $K$-theory infinite loop space. 
A Waldhausen category which satisfies additional axioms and/or has
additional structure 
will have better properties which e.g. make the problem of actually
computing the associated $K$-theory more tractable.
For example, a Waldhausen category satisfying the ``extension axiom''
and the ``saturation axiom'' and equipped with an additional structure
called a ``cylinder functor'' admits Waldhausen's {\em Localization Theorem}
(see \cite{MR802796}), a computationally powerful result 
that describes the sense in which
localizations of the Waldhausen category induce long exact sequences in 
the $K$-theory groups.

Meanwhile, in stable representation theory, one regards the projective
modules over a ring as ``contractible,'' and maps of modules that factor through
projective modules are regarded as ``nulhomotopic.'' This suggests
that the category of modules over a ring perhaps has a Waldhausen category
structure in which the objects weakly equivalent to zero---the contractible
objects---are precisely the projective modules.
More generally, one has the tools of {\em relative homological algebra}: if
one chooses a sufficiently well-behaved class of objects in an abelian category,
one can do a form of homological algebra in which the chosen class of objects
plays the role of projective objects. Given an abelian category and a class 
of relative projective objects, one wants to know if there is a 
natural Waldhausen category structure on that abelian category, such
that the contractible objects are precisely the relative projectives.
One also wants to know how many extra axioms are satisfied by, and how much additional structure is admitted by, such a Waldhausen category.

In this paper we prove the following theorems that answer the above questions, and
explain fundamental properties of the relationships between Waldhausen 
$K$-theory, relative homological algebra, and stable representation theory:
\begin{enumerate}
\item Definition-Proposition~\ref{waldhausen cat from an allowable class}:
Given an abelian category $\mathcal{C}$ and a 
sufficiently nice pair of allowable classes $E,F$ in $\mathcal{C}$, there exists
a Waldhausen category structure on $\mathcal{C}$ whose weak equivalences
are the $E$-stable equivalences and whose cofibrations are the
$F$-monomorphisms. In particular, the $E$-projective objects are precisely
the contractible objects in this Waldhausen category. This Waldhausen
category satisfies the saturation axiom and the extension axiom.
\item Theorem~\ref{existence of cylinder functors}:
If $\mathcal{C}$ has enough injectives, then $\mathcal{C}$ has a cylinder functor
satisfying the cylinder axiom if and only if $\mathcal{C}$ obeys a certain
generalized quasi-Frobenius condition: every object must functorially embed in
an $E$-projective object by an $F$-monomorphism.
\item As a consequence, we have Corollary~\ref{main cor}: any quasi-Frobenius
abelian category with enough projectives and functorially enough
injectives admits the structure of a Waldhausen category whose
weak equivalences are the stable equivalences and whose 
cofibrations are the monomorphisms. This Waldhausen category
has a cylinder functor, and it satisfies the saturation, extension, and
cylinder axioms.
\item As a consequence, we have Corollary~\ref{algebraic main cor}:
if $R$ is a finite-dimensional quasi-Frobenius algebra over a
finite field, then the category of finitely generated (left) $R$-modules
admits the structure of a Waldhausen
category in which the cofibrations are the monomorphisms and the
weak equivalences are the stable equivalences. 
This Waldhausen category satisfies the 
saturation and extension axioms, and it admits a cylinder functor
satisfying the cylinder axiom.
\item In Proposition~\ref{existence of multiplicative structure on stable g-thy}
we show that, for a finite-dimensional co-commutative Hopf algebra over a finite field,
this Waldhausen category is even better: it has a multiplicative structure
coming from the tensor product of modules over the base field,
and this multiplicative structure gives rise to the structure of a 
homotopy-commutative 
ring spectrum on the Waldhausen $K$-theory spectrum of the category.
\item We use the theorems described above to prove our
Proposition~\ref{stable G is delooping}, in which we we show that,
for finite-dimensional quasi-Frobenius nilpotent extensions of
finite fields, stable $G$-theory is a delooping of relative
algebraic $K$-theory. 
\item As a consequence, we get our 
Theorem~\ref{main thm computing Gst}, in which we prove that, under 
the same assumptions on the ring, stable $G$-theory in positive degrees vanishes
$\ell$-adically and is isomorphic to topological cyclic homology with a degree
shift $p$-adically, where $p$ is the characteristic of the base field.
\item Finally, in Theorem~\ref{main thm, multiplicative version},
we compute the homotopy groups of the connective cover of
the stable $G$-theory spectrum $g_{st}(\mathbb{F}_{p}[x]/x^{p^n})$
of truncated polynomial algebras over $\mathbb{F}_p$, in terms
of the Hesselholt-Madsen description of the topological cyclic
homology of truncated polynomial algebras. We also provide one possible answer
to a question (unpublished) of J. Morava: under what circumstances
does algebraic $K$-theory, or a localization or other modification thereof, admit
the structure of a complex oriented ring spectrum?
Our answer is that the connective stable $G$-theory spectrum
$g_{st}(\mathbb{F}_{p}[x]/x^{p^n})$ is, for all $p$ and $n$, a complex oriented
ring spectrum.
\end{enumerate}

``Quasi-Frobenius conditions'' appear prominently throughout this paper. 
Recall that a ring $R$ is said to be quasi-Frobenius if every projective
$R$-module is injective and vice versa.
The appearance of these conditions in connection with Waldhausen $K$-theory
stems from the theorem of Faith and Walker
(see \cite{MR6728009} for a good account of this and related theorems):
\begin{theorem} {\bf (Faith-Walker.)}\label{faith-walker thm}
A ring $R$ is quasi-Frobenius if and only if every $R$-module embeds in a projective $R$-module.
\end{theorem}

Here is one point of view on the significance of Theorem~\ref{existence of cylinder functors}. It
has been known for a long time, e.g. as described in \cite{MR1650134}, that 
when $R$ is a quasi-Frobenius ring, there exists a model category structure on the category of $R$-modules
in which the cofibrations are the injections and the weak equivalences are the stable equivalences of modules. 
Constructing this
model category structure uses the quasi-Frobenius condition in an essential way. But a Waldhausen category
structure on $R$-modules is weaker, less highly-structured, than a model category structure; so one might
hope that, even in the absence of the quasi-Frobenius condition on $R$, one could put the structure of a 
Waldhausen category on $R$-modules, such that the cofibrations are the injections and the 
weak equivalences are
the stable equivalences of modules. As a consequence of Theorem~\ref{existence of cylinder functors}, one
only gets a Waldhausen category structure with cylinder functor on the category of $R$-modules if $R$
is quasi-Frobenius. So {\em the category of $R$-modules admits a model category structure as desired
if and only if it admits a Waldhausen category structure with cylinder functor as desired.}
(But our results, such as Corollary~\ref{main cor} on existence of the cylindrical Waldhausen category structure, also have the virtue of applying to
quasi-Frobenius abelian categories that are not categories of modules over a
ring.)

So one knows that, when $R$ is a quasi-Frobenius 
ring, then one has the model category of 
$R$-modules with cofibrations inclusions and weak equivalences the stable
equivalences, and from the theorems in this paper which we have described above,
one knows that relaxing the condition that $R$ be quasi-Frobenius does not
enable one to get cylindrical Waldhausen category structures
in any greater generality. Then one asks the natural question: restricting
to the finitely-generated $R$-modules, {\em what are the $K$-groups of
this Waldhausen category}? That leads us to the computations of
Theorem~\ref{main thm computing Gst} and Theorem~\ref{main thm, multiplicative version}.

We remark that our stable $G$-theory Waldhausen category poses an alternative
to a construction by G. Garkusha in \cite{MR1932156}, who constructs a Waldhausen
category which models the cofiber (on the spectrum level) of the Cartan map
from $K$-theory to $G$-theory. Garkusha's construction is a Waldhausen
category structure on chain complexes of $R$-modules, and when $R$ is
quasi-Frobenius, 
our Proposition~\ref{relation between g and k} also describes the cofiber of the Cartan
map but with a much smaller model than Garkusha's 
(the stable $G$-theory Waldhausen category
structure on $R$-modules, rather than on chain complexes of $R$-modules). 
Our stable $G$-theory, as a model for the cofiber of the Cartan map, also 
has the advantage of multiplicative structure, as in 
Proposition~\ref{existence of multiplicative structure on stable g-thy}.

We would not have written this paper or thought about any of these issues
if not for conversations we had with Crichton Ogle, who taught us
a great deal about Waldhausen $K$-theory during the summer of 2012.
We are grateful to C. Ogle for his generosity in teaching us about this subject.

\section{Waldhausen category structures from allowable classes on abelian categories.}

\subsection{Definitions.}

This subsection, mostly consisting of definitions, 
is entirely review and there
are no new results or definitions in it, with the exception of
Definition~\ref{def of retractile monics} and Definition-Proposition~\ref{definition of closures}.

Throughout this subsection, let $\mathcal{C}$ be an 
abelian category.

We begin with the definition of an allowable class. An allowable class is
the structure one needs to specify on $\mathcal{C}$ in order to have a
notion of relative homological algebra in $\mathcal{C}$.
\begin{definition} 
An {\em allowable class in $\mathcal{C}$} consists of a collection 
$E$ of short exact sequences in $\mathcal{C}$ which is closed under isomorphism
of short exact sequences and which contains every short exact sequence
in which at least one object is the zero object of $\mathcal{C}$.
(See section IX.4 of \cite{MR1344215} for this definition
and basic properties.)
\end{definition}
The usual ``absolute'' homological algebra in an abelian category 
$\mathcal{C}$ is recovered by letting the allowable class $E$ consist
of {\em all} short exact sequences in $\mathcal{C}$.

Once one chooses an allowable class $E$, one has the notion of monomorphisms
relative to $E$, or ``$E$-monomorphisms,'' and epimorphisms relative to $E$,
or ``$E$-epimorphisms.''
\begin{definition}
Let $E$ be an allowable class in $\mathcal{C}$.
A monomorphism $f: M\rightarrow N$ 
in $\mathcal{C}$ is called an {\em $E$-monomorphism} or an {\em $E$-monic}
if the short exact sequence
\[ 0\rightarrow M\stackrel{f}{\longrightarrow} N \rightarrow \coker f\rightarrow 0\]
is in $E$.

Dually, an epimorphism $g: M\rightarrow N$ is called an 
{\em $E$-epimorphism} or an {\em $E$-epic} if the short exact sequence
\[ 0\rightarrow \ker f\rightarrow M\stackrel{f}{\longrightarrow} N \rightarrow 0\]
is in $E$.\end{definition}
In the absolute case, the case that $E$ is all short exact sequences in $\mathcal{C}$,
the $E$-monomorphisms are simply the monomorphisms, and the $E$-epimorphisms
are simply the epimorphisms.

Projective and injective objects are at the heart of homological algebra. In 
relative homological algebra, one has the notion of relative projectives, or
$E$-projectives: these are simply the objects which lift over every 
$E$-epimorphism. The $E$-injectives are defined dually.
\begin{definition}
Let $E$ be an allowable class in $\mathcal{C}$.
An object $X$ of $\mathcal{C}$ is said to be an {\em $E$-projective}
if, for every diagram
\[ \xymatrix{ & X \ar[d]  \\
 M \ar[r]^f & N }\]
in which $f$ is an $E$-epic,
there exists a morphism $X\rightarrow M$ making the above diagram commute.

Dually, 
an object $X$ of $\mathcal{C}$ is said to be an {\em $E$-injective}
if, for every diagram
\[ \xymatrix{ M\ar[r]^f\ar[d] & N   \\
X & }\]
in which $f$ is an $E$-monic,
there exists a morphism $N\rightarrow X$ making the above diagram commute.

When the allowable class $E$ is clear from context we sometimes refer to $E$-projectives
and $E$-injectives as {\em relative projectives} and {\em relative injectives,} respectively.
\end{definition}
In the absolute case, the case that $E$ is all short exact sequences 
in $\mathcal{C}$, the $E$-projectives are simply the projectives, and 
the $E$-injectives are simply the injectives.

Once one has a notion of relative projectives, one has a reasonable notion of
a stable equivalence or, loosely, a ``homotopy'' between maps, as studied
(usually in the absolute case, where $E$-projectives are simply projectives) in stable representation theory:
\begin{definition}
Let $E$ be an allowable class in $\mathcal{C}$.
Let $f,g: M\rightarrow N$ be morphisms in $\mathcal{C}$.
We say that $f$ and $g$ are {\em $E$-stably equivalent}
and we write $f\simeq g$
if $f-g$ factors through an $E$-projective object of 
$\mathcal{C}$. 
\end{definition}

One then has the notion of stable equivalence of objects, or loosely,
``homotopy equivalence'':
\begin{definition}\label{def of stable equivalence}
We say that a map $f: M\rightarrow N$ is a
{\em $E$-stable equivalence} if there exists a map
$h: N\rightarrow M$ such that
$f\circ h\simeq \id_N$ and $h\circ f\simeq \id_M$.
\end{definition}
In the absolute case where $E$ consists of all short exact
sequences in $\mathcal{C}$, this is the usual notion of stable
equivalence of modules over a ring. 
Over a Hopf algebra over a field,
stably equivalent modules have isomorphic cohomology in 
positive degrees, so if one is serious about computing
the cohomology of all finitely-generated modules over a
particular Hopf algebra---such as the Steenrod algebra or
the group ring of a Morava stabilizer group---it 
is natural to first compute the 
representation ring modulo stable equivalence.
See \cite{MR738973} for this useful 
perspective (which motivates much of the work in this paper).

We now define the relative-homological-algebraic
generalizations of an abelian categories having
enough projectives or enough injectives. We provide an extra
twist on this definitions as well, which we will need for 
certain theorems: sometimes we will need to know that,
for example, not only does every object embed in an injective,
but that we can choose such embeddings in a functorial way.
\begin{definition} 
Let $E$ be an allowable class in $\mathcal{C}$.
We say that {\em $\mathcal{C}$ has enough $E$-projectives}
if, for any object $M$ of $\mathcal{C}$, there exists an $E$-epic
$N\rightarrow M$ with $N$ an $E$-projective.
We say that {\em $\mathcal{C}$ has functorially enough $E$-projectives}
if $\mathcal{C}$ has enough $E$-projectives and the choice of 
$E$-epimorphisms from $E$-projectives to each object of $\mathcal{C}$
can be made functorially, i.e., there exists a functor
$P: \mathcal{C}\rightarrow\mathcal{C}$ together with a natural transformation
$\epsilon: P\rightarrow \id_{\mathcal{C}}$ such that
$P(X)$ is $E$-projective and $\epsilon(X): P(X)\rightarrow X$ is
an $E$-epimorphism for all objects $X$ of $\mathcal{C}$, and such that,
if $f:X\rightarrow Y$ is an $E$-epimorphism, then so is
$P(f): P(X)\rightarrow P(Y)$.

Dually, we say that {\em $\mathcal{C}$ has enough $E$-injectives}
if, for any object $M$ of $\mathcal{C}$, there exists an $E$-monic
$M\rightarrow N$ with $N$ an $E$-injective.
We say that {\em $\mathcal{C}$ has functorially enough $E$-injectives}
if $\mathcal{C}$ has enough $E$-injectives and the choice of 
$E$-monomorphisms from $E$-injectives to each object of $\mathcal{C}$
can be made functorially, i.e., there exists a functor
$I: \mathcal{C}\rightarrow\mathcal{C}$ together with a natural transformation
$\eta: \id_{\mathcal{C}}\rightarrow I$ such that
$I(X)$ is $E$-injective and $\eta(X): X\rightarrow I(X)$ is
an $E$-monomorphism for all objects $X$ of $\mathcal{C}$, and such that,
if $f:X\rightarrow Y$ is an $E$-monomorphism, then so is
$I(f): I(X)\rightarrow I(Y)$.
\end{definition}
Our need to have abelian categories with {\em functorially}
enough injectives or projectives is only due to Waldhausen's
definitions of cylinder functors and resulting theorems, 
in \cite{MR802796}, demanding that
cylinder functors actually be {\em functors.} It seems
likely that one can do away with this assumption and still 
prove analogues of Waldhausen's results that use cylinders
(e.g. the Fibration Theorem) by mimicking the situation in 
model category theory: there, one knows that any morphism
has a factorization into a cofibration followed by an acylic
fibration, but such factorizations are not provided in a 
functorial way. We don't pursue that angle in this paper, however.

Finally, we have our first definition of a quasi-Frobenius
condition:
\begin{definition}
Let $E$ be an allowable class in $\mathcal{C}$.
We will call $E$ a {\em quasi-Frobenius allowable class} if the 
$E$-projectives are exactly the $E$-injectives.
If the allowable class consisting of all short exact sequences in $\mathcal{C}$
is a quasi-Frobenius class, then we will simply say that
{\em $\mathcal{C}$ is quasi-Frobenius.}
\end{definition}

Here are some important examples of allowable classes in abelian categories:
\begin{itemize}
\item
As described above, the 
usual ``absolute'' homological algebra in an abelian category 
$\mathcal{C}$ is recovered by letting the allowable class $E$ consist
of {\em all} short exact sequences in $\mathcal{C}$; then
the $E$-projectives are the usual projectives, etc. Note that, if
$E$ is an arbitrary allowable class in $\mathcal{C}$, then any
projective object is an $E$-projective object, but the converse is
not necessarily true.
\item 
There is another ``trivial'' case of an allowable class: 
if we let $E$ consist of only the short exact sequences in $\mathcal{C}$ in which
at least one of the objects is the zero object, then the $E$-epics
consists of all identity maps as well as all projections to the zero object,
and the $E$-monics consist of all identity maps as well as all inclusions
of the zero object. As a consequence all objects are both $E$-injectives 
and $E$-projectives, and $E$ is a quasi-Frobenius allowable class.
\item 
Suppose $\mathcal{C},\mathcal{D}$ are abelian categories and
$F: \mathcal{C}\rightarrow\mathcal{D}$ is an additive functor. 
Then we can let $E$ be the allowable class in $\mathcal{C}$
consisting of the short exact sequences which are sent by $F$ to split short exact sequences in $\mathcal{D}$. 
If $F$ has a left (resp. right) adjoint $G$ then
objects of $\mathcal{C}$ of the form $GFX$ (resp. $FGX$) are $E$-projectives 
(resp. $E$-injectives) and the counit map $GFX\rightarrow X$ of
the comonad $GF$ (resp. the unit map $X\rightarrow GFX$ of the monad $GF$)
is an $E$-epic (resp. $E$-monic), hence $\mathcal{C}$ has
enough $E$-projectives (resp. enough $E$-injectives). 
These ideas are in \cite{MR1344215}.

For example, if $R$ is a ring and $\mathcal{C}$ the category of $R$-modules and
$\mathcal{D}$ the category of abelian groups, and $F$ the forgetful functor,
then $E$ is the class of short exact sequences of $R$-modules whose underlying
short exact sequences of abelian groups are split. The $R$-modules of the form
$R\otimes_{\mathbb{Z}} M$, for $M$ an $R$-module, are $E$-projectives. 
\end{itemize}

Here is a new definition which makes many arguments involving allowable classes substantially smoother:
\begin{definition} \label{def of retractile monics}
An allowable class $E$ is said to 
{\em have retractile monics}
if, whenever $g\circ f$ is an $E$-monic, $f$ is also an $E$-monic.

Dually, an allowable class $E$ is said to 
{\em have sectile epics}
if, whenever $g\circ f$ is an $E$-epic, $g$ is also an $E$-epic.
\end{definition}

The utility of the notion of ``having sectile epics'' comes from the following fundamental theorem of
relative homological algebra, due to Heller (see \cite{MR1344215}), whose statement is slightly cleaner is one is willing to use the phrase ``having sectile epics.'' The consequence
of Heller's theorem is that, in order to specify a ``reasonable'' allowable class
in an abelian category, it suffices to specify the relative projective objects
associated to it.
\begin{theorem}\label{heller's theorem}
If $\mathcal{C}$ is an abelian category and $E$ is an allowable class in $\mathcal{C}$ with sectile epics and enough $E$-projectives,
then an epimorphism $M\rightarrow N$ in $\mathcal{C}$ is an $E$-epic if and
only if the induced map $\hom_{\mathcal{C}}(P, M)\rightarrow 
\hom_{\mathcal{C}}(P, N)$ of abelian groups is an epimorphism for all
$E$-projectives $P$.
\end{theorem}

Heller's theorem suggests the following 
construction, which as far as we know, is new 
(but unsurprising): 
if $E$ is an allowable class, we can construct a
``sectile closure of $E$'' which
has the same relative projectives and the same
stable equivalences but which has sectile 
epics. Here are the specific properties
of this construction (we have neglected to
write out proofs of these properties 
because the proofs are so elementary):
\begin{definition-proposition}\label{definition of closures}
Let $\mathcal{C}$ be an abelian category,
$E$ an allowable class in $\mathcal{C}$.
Let $E_{sc}$ be the allowable class
in $\mathcal{C}$ consisting of 
the exact sequences
\[ X \rightarrow Y \rightarrow Y/X\]
such that the induced map
\[ \hom_{\mathcal{C}}(P, Y)\rightarrow
\hom_{\mathcal{C}}(P, Y/X)\]
is a surjection of abelian groups for every
$E$-projective $P$. We call
$E_{sc}$ the {\em sectile closure} of $E$.
The allowable class $E_{sc}$ has the following
properties:
\begin{itemize}
\item
$E_{sc}$ has sectile epics.
\item 
An object of $\mathcal{C}$ is an $E$-projective 
if and only if it is an
$E_{sc}$-projective.
\item 
If $f,g$ are two morphisms in $\mathcal{C}$ then
$f$ and $g$ are $E$-stably equivalent
if and only if they are $E_{sc}$-stably equivalent.
\item 
If $X,Y$ are two objects in $\mathcal{C}$
then $X$ and $Y$ are $E$-stably equivalent
if and only if they are $E_{sc}$-stably equivalent.
\item 
$(E_{sc})_{sc} = E_{sc}$.
\item
If $E,F$ are two allowable classes in $\mathcal{C}$
and $F\subseteq E$ then $F_{sc}\subseteq E_{sc}$.
\end{itemize}
\end{definition-proposition}

Of course there is a construction dual
to the sectile closure,
a {\em retractile closure}, with dual properties,
but with a less straightforward relationship
to stable equivalence, since stable equivalence
is defined in terms of projectives, not injectives.

We now recall Waldhausen's definitions:
\begin{definition} \label{def of waldhausen cat}
A pointed category $\mathcal{C}$ with finite pushouts equipped with a specified class of cofibrations and a specified class of weak equivalences, both closed under composition, is called a {\em Waldhausen
category} if
the following axioms are satisfied:
\begin{itemize}
\item {\bf (Cof 1.)} The isomorphisms in $\mathcal{C}$
are cofibrations.
\item {\bf (Cof 2.)} For every object $X$ of
$\mathcal{C}$, the map $\pt \rightarrow \mathcal{C}$
is a cofibration. (We write $\pt$ for the zero
object of $\mathcal{C}$.)
\item {\bf (Cof 3.)} If $X\rightarrow Y$ is a morphism
in $\mathcal{C}$ and $X \rightarrow Z$ is a cofibration,
then the canonical map $Y\rightarrow Y\coprod_X Z$
is a cofibration.
\item {\bf (Weq 1.)} The isomorphisms in $\mathcal{C}$
are weak equivalences.
\item {\bf (Weq 2.)} If 
\begin{equation}\label{weq 2 diagram}  \xymatrix{ 
Y \ar[d] & X \ar[l]\ar[r]\ar[d] & Z\ar[d] \\
Y^{\prime} & X^{\prime} \ar[l]\ar[r] & Z^{\prime} }\end{equation}
is a commutative diagram in $\mathcal{C}$
in which the maps $X\rightarrow Y$ and
$X^{\prime}\rightarrow Y^{\prime}$ are
cofibrations and all three vertical maps are
weak equivalences, then the induced map
$Y\coprod_X Z\rightarrow Y^{\prime}\coprod_{X^{\prime}} Z^{\prime}$
is a weak equivalence.
\end{itemize}
\end{definition}
Ultimately, if $\mathcal{C}$ is a Waldhausen category, then what one typically
wants to understand is $\left| wS_{\cdot} \mathcal{C}\right|$, the 
geometric realization of the simplicial category 
$wS_{\cdot}\mathcal{C}$ constructed by Waldhausen in \cite{MR802796}.
The $K$-groups of $\mathcal{C}$ are defined as the homotopy groups of the
loop space $\Omega \left| wS_{\cdot} \mathcal{C}\right|$:
\begin{eqnarray*} \pi_{n+1}(\left| wS_{\cdot} \mathcal{C}\right|) & \cong &
 \pi_n(\Omega \left| wS_{\cdot} \mathcal{C}\right|) \\
 & \cong & K_n(\mathcal{C}).\end{eqnarray*}

If $\mathcal{C}$ is a Waldhausen category then
the following axioms may or may not be satisfied:
\begin{definition}\label{def of sat and ext axioms}\begin{itemize}
\item {\bf (Saturation axiom.)} If $f,g$ are composable
maps in $\mathcal{C}$ and two of $f,g,g\circ f$
are weak equivalences then so is the third.
\item {\bf (Extension axiom.)} If 
\[ \xymatrix{ 
X \ar[r]\ar[d] & Y \ar[r]\ar[d] & Y/X\ar[d] \\
X^{\prime} \ar[r] & Y^{\prime} \ar[r] & 
 Y^{\prime}/X^{\prime} }\]
is a map of cofiber sequences and
the maps $X\rightarrow X^{\prime}$
and $Y/X\rightarrow Y^{\prime}/X^{\prime}$
are weak equivalences then so is
the map $Y\rightarrow Y^{\prime}$.
\end{itemize}\end{definition}

\begin{definition} 
If $\mathcal{C}$ is a Waldhausen category,
a {\em cylinder functor on $\mathcal{C}$} is a functor
from the category of arrows $f: X\rightarrow Y$
in $\mathcal{C}$ to
the category of diagrams of the form
\[\xymatrix{
X\ar[r]^{j_1} \ar[rd]^f & T(f) \ar[d]^p & Y\ar[l]_{j_2}\ar[ld]^{\id} \\
 & Y & 
}\] 
in $\mathcal{C}$ satisfying the three conditions:
\begin{itemize}
\item {\bf (Cyl 1.)} 
If 
\[ \xymatrix{ X^{\prime}\ar[r]^{f^{\prime}} \ar[d] & Y^{\prime} \ar[d] \\
 X\ar[r]^f & Y}\]
is a commutative diagram in $\mathcal{C}$ in which
the vertical maps are weak equivalences (resp. cofibrations),
then the map $T(f^{\prime})\rightarrow T(f)$ is a weak equivalence (resp. cofibration
and 
\[ X \coprod_{X^{\prime}} T(f^{\prime}) \coprod_{Y^{\prime}} Y \rightarrow T(f)\]
is a cofibration).
\item {\bf (Cyl 2.)} 
$T(\pt\rightarrow Y) = Y$ and the maps $p$ and $j_2$
are the identity map on $Y$.
\item {\bf (Cyl 3.)}
The map $j_1\coprod j_2: A\coprod B \rightarrow T(f)$ is a cofibration.
\end{itemize}
A Waldhausen category with cylinder functor
may or may not satisfy the additional
condition:
\begin{itemize}
\item {\bf (Cylinder axiom.)} 
For any map $f$ in $\mathcal{C}$,
the map $p$ is a weak equivalence.
\end{itemize}
\end{definition}
The idea here is that a cylinder functor satisfying the cylinder axiom acts
very much like the mapping cylinder construction from classical homotopy 
theory---or, more generally, like fibrant replacement in a model category.
Some of Waldhausen's most powerful results in \cite{MR802796}
have proofs of a sufficiently homotopy-theoretic flavor that they require
that every Waldhausen category in sight has a cylinder functor obeying
the cylinder axioms. A good example of this is Waldhausen's Fibration 
Theorem, which we now recall:
\begin{theorem} {\bf Fibration Theorem (Waldhausen).}  \label{fibration thm}
Suppose $\mathcal{C},\mathcal{C}_0$ are Waldhausen categories
with the same underlying category and the same underlying class of cofibrations.
Suppose all of the following conditions are satisfied:
\begin{itemize}
\item Every weak equivalence in $\mathcal{C}$ is also
a weak equivalence in $\mathcal{C}_0$.
\item $\mathcal{C}_0$ admits a cylinder functor satisfying the cylinder axiom.
\item The weak equivalences in $\mathcal{C}_0$ 
satisfy the saturation and extension axioms.
\end{itemize}
Then
\[ \left| wS_{\cdot}\mathcal{X}\right| \rightarrow \left| wS_{\cdot}\mathcal{C}\right| 
\rightarrow \left| wS_{\cdot}\mathcal{C}_0\right| \]
is a homotopy fibre sequence, where $\mathcal{X}$ is the full 
sub-Waldhausen-category of $\mathcal{C}$ generated by the objects that are
weakly equivalent to $\pt$ in $\mathcal{C}_0$.
As a consequence, after looping and taking homotopy groups, we get the long
exact sequence of $K$-groups:
\[ \dots \rightarrow K_{n+1}(\mathcal{C}_0)\rightarrow K_n(\mathcal{X}) \rightarrow
K_n(\mathcal{C})\rightarrow K_n(\mathcal{C}_0)\rightarrow K_{n-1}(\mathcal{X})
\rightarrow \dots .\]
\end{theorem}
The question of when our Waldhausen categories given by allowable classes
on abelian categories satisfy the required conditions for the Fibration Theorem
to hold is the subject of most of this paper.

\subsection{The Waldhausen category structure on an abelian category associated to a pair of allowable classes.}

In this subsection,
we will prove that an abelian category equipped with a pair of allowable 
classes $E,F$ admits a Waldhausen category structure in which 
the cofibrations are the $F$-monomorphisms and the 
weak equivalences are the $E$-stable equivalences. In this subsection we will
also find conditions under which this Waldhausen category satisfies
the extension and saturation axioms (see Definition~\ref{def of sat and ext axioms}
for definitions of these axioms). To get anywhere, we will need some lemmas:

\begin{lemma}\label{pullback of surjections is surjective}
A pullback of a surjective map of abelian groups is surjective.
\end{lemma}
\begin{proof}
The forgetful functor from abelian groups to sets is a right adjoint, hence preserves limits.
It also clearly preserves surjections.
So the lemma is true if a pullback of a surjective maps of sets is surjective, which is 
an elementary exercise.
\end{proof}

\begin{lemma} 
\label{when admissible monics are closed under pushout}
Let $\mathcal{C}$ be an abelian category and let
$E$ be an allowable class with retractile
monics. Then $E$-monics
are closed under pushout in $\mathcal{C}$. That is,
if $X\rightarrow Z$ is an $E$-monic and $X\rightarrow Y$ is any morphism
in $\mathcal{C}$, then the canonical
map $Y\rightarrow Y\coprod_X Z$ is
an $E$-monic.
\end{lemma}
\begin{proof}
Suppose $f: X\rightarrow Z$ is an $E$-monic and $X\rightarrow Y$ any morphism. 
We have the commutative diagram with exact rows
\[ \xymatrix{
0 \ar[r] & 
 X \ar[r]^f \ar[d] & 
 Z \ar[r]\ar[d] &
 \coker f \ar[r] \ar[d] &
 0 \ar[d] \\
 &
Y \ar[r] &
Y\coprod_X Z \ar[r] &
\coker f \ar[r] & 
 0 }\]
and hence, for every $E$-injective $I$, 
the induced commutative diagram
of abelian groups
\[ \xymatrix{
0 \ar[r] \ar[d] & 
 \hom_{\mathcal{C}}(\coker f, I) \ar[r] &
 \hom_{\mathcal{C}}(Z, I) \ar[r] &
 \hom_{\mathcal{C}}(X, I) \ar[r] &
 0 \\
0 \ar[r]  & 
 \hom_{\mathcal{C}}(\coker f, I) \ar[r]\ar[u]^{\cong} &
 \hom_{\mathcal{C}}(Y\coprod_X Z, I) \ar[r]\ar[u] &
 \hom_{\mathcal{C}}(Y, I). \ar[u] & }\]
Exactness of the top row follows from $f$
being an $E$-monic together with
$E$ having retractile monics, hence
$E$ is its own retractile closure, hence 
$E$-monics are precisely the
maps which induce a surjection after applying
$\hom_{\mathcal{C}}(- ,I)$ for every $E$-injective $I$.
Now in particular we have a commutative square
in the above commutative diagram:
\[ 
\xymatrix{
 \hom_{\mathcal{C}}(Z, I) \ar[r] &
 \hom_{\mathcal{C}}(X, I)  \\
 \hom_{\mathcal{C}}(Y\coprod_X Z, I) \ar[r]\ar[u] &
 \hom_{\mathcal{C}}(Y, I), \ar[u]
}\]
which is a pullback square of abelian groups,
by the universal property of the pushout.
The top map in the square is a surjection,
hence so is the bottom map, Lemma~\ref{pullback of surjections is surjective}.
So $\hom_{\mathcal{C}}(Y\coprod_X Z, I)
\rightarrow \hom_{\mathcal{C}}(Y, I)$
is a surjection for every $E$-injective $I$. Again since
$E$ is its own retractile closure, this implies
that $Y\rightarrow Y\coprod_X Z$ is an $E$-monic.
\end{proof}

\begin{lemma}\label{composite of E-monics is E-monics}
Suppose $\mathcal{C}$ is an abelian category, $E$ an allowable class in $\mathcal{C}$ with retractile monics.
Suppose $\mathcal{C}$ has enough $E$-injectives.
A composite of $E$-monomorphisms is an $E$-monomorphism.
\end{lemma}
\begin{proof}
Let $f:X\rightarrow Y$ and $g: Y \rightarrow Z$ be $E$-monomorphisms. Let $I$ be an $E$-injective object equipped with a map
$X \rightarrow I$. Then, since $f$ is an $E$-monomorphism, $X\rightarrow I$ extends through $f$ to a map $Y \rightarrow I$, which in turn extends through
$g$ since $g$ is an $E$-monomorphism. So every map to an $E$-injective from $Z$ extends through $g\circ f$. Now, by the dual of 
Heller's theorem~\ref{heller's theorem}, $g\circ f$ is an $E$-monomorphism.
\end{proof}

\begin{lemma} \label{composite of stable equivalences is a stable equivalence}
Let $\mathcal{C}$ be an abelian category and 
let $E$ be an 
allowable class in $\mathcal{C}$.
Then a composite of two $E$-stable equivalences in $\mathcal{C}$ is
an $E$-stable equivalence.
\end{lemma}
\begin{proof}
Let $X\stackrel{f}{\longrightarrow} Y\stackrel{g}{\longrightarrow} Z$
be a pair of $E$-stable equivalences.
Then there exist $E$-projective objects $P_X,P_Y$ of $\mathcal{C}$ 
and morphisms 
\begin{eqnarray*} 
Y & \stackrel{f^{\prime}}{\longrightarrow} & X, \\
Z & \stackrel{g^{\prime}}{\longrightarrow} & Y, \\
X & \stackrel{i_X}{\longrightarrow} & P_X, \\
Y & \stackrel{i_Y}{\longrightarrow} & P_Y, \\
P_X & \stackrel{s_X}{\longrightarrow} & X, \mbox{\ and} \\
P_Y & \stackrel{s_Y}{\longrightarrow} & Y \end{eqnarray*}
such that 
\begin{eqnarray*} \label{eqn 1} f^{\prime}\circ f - \id_X & = & s_X\circ i_X \mbox{\ and} \\
\label{eqn 2} g^{\prime}\circ g - \id_Y & = & s_Y \circ i_Y.\end{eqnarray*}
Then we have
\[ f^{\prime}\circ g^{\prime}\circ g \circ f - \id_X = 
s_x\circ i_X + f^{\prime}\circ s_Y \circ i_Y\circ f\]
so $f^{\prime}\circ g^{\prime}\circ g \circ f - \id_X$
factors through the $E$-projective $P_X\oplus P_Y$.
A similar argument applies to showing
that $g\circ f\circ f^{\prime} \circ g^{\prime} - \id_Z$
factors through an $E$-projective. So 
$g\circ f$ is an $E$-stable equivalence.
\end{proof}

\begin{definition-proposition}
\label{waldhausen cat from an allowable class}
Let $\mathcal{C}$ be an abelian category, let
$E,F$ be allowable classes in $\mathcal{C}$ with $F\subseteq E$.
Suppose each of the following conditions are satisfied:
\begin{itemize}
\item
$F$ has retractile monics.
\item
$E$ has retractile monics and sectile epics.
\item 
$\mathcal{C}$ has enough $F$-injectives.
\item 
$\mathcal{C}$ has enough $E$-projectives and enough $E$-injectives.
\item 
Every $E$-projective object is $E$-injective.
\end{itemize}
Then there exists a 
Waldhausen category structure on $\mathcal{C}$ in 
which the cofibrations are the $F$-monomorphisms
and the weak equivalences are the 
$E$-stable equivalences.
We write $\mathcal{C}_{E-we}^{F-cof}$ for this Waldhausen category.
This Waldhausen category satisfies the
saturation axiom and the extension axiom.
\end{definition-proposition}
\begin{proof}
We check Waldhausen's axioms from 
Definition~\ref{def of waldhausen cat}.
In the case of an abelian category $\mathcal{C}$ 
and allowable classes $E,F$
with the stated 
classes of cofibrations and weak equivalences,
axioms (Cof 1) and (Cof 2) and (Weq 1) are immediate.
That the class of cofibrations is closed under composition is Lemma~\ref{composite of E-monics is E-monics}.
That the class of weak equivalences is closed under composition is Lemma~\ref{composite of stable equivalences is a stable equivalence}.
We show that the remaining axioms are satisfied:
\begin{itemize}
\item Axiom (Cof 3) is a consequence of Lemma~\ref{when admissible monics are closed under pushout}.
\item Axiom (Weq 2) is actually fairly substantial and takes some
work to prove---enough so that we moved this work into a paper of its own,
\cite{rectifications}. In Corollary~4.4 of that paper, we prove that,
if $E=F$, $\mathcal{C}$ has enough $E$-projectives and enough
$E$-injectives, $E$ has sectile epics and retractile monics,
and every $E$-projective object is $E$-injective, then
$\mathcal{C}$ satisfies axiom (Weq 2). We refer the reader to 
that paper for the proof, which requires some work and a number of preliminary 
results, and would make the present paper much longer if we included it here.

Once we have the result for $F=E$, it follows for $F\subseteq E$, since if
diagram~\ref{weq 2 diagram} has its horizontal maps $F$-monomorphisms,
the horizontal maps are also then $E$-monomorphisms.
\item
The saturation axiom follows easily from Lemma~\ref{composite of stable equivalences is a stable equivalence} together with the observation that,
if $X\stackrel{f}{\longrightarrow} Y$ is a $E$-stable equivalence and $Y\stackrel{f^{\prime}}{\longrightarrow} X$ is a morphism such that
$f\circ f^{\prime} - \id_Y$ and $f^{\prime}\circ f - \id_X$ both factor through $E$-projective objects, then $f^{\prime}$ is also a $E$-stable equivalence.
That is, $E$-stable equivalences have ``up-to-equivalence inverses.''
\item
Finally, we handle the extension axiom. We begin by assuming that $E=F$, and that
\begin{equation}\label{ses diagram} \xymatrix{ 
X \ar[r]\ar[d]^f & Y \ar[r]\ar[d]^g & Y/X\ar[d]^h \\
X^{\prime} \ar[r] & Y^{\prime} \ar[r] & 
 Y^{\prime}/X^{\prime} }\end{equation}
is a map of cofiber sequences (i.e., short exact sequences in $E$) and
the maps $X\rightarrow X^{\prime}$
and $Y/X\rightarrow Y^{\prime}/X^{\prime}$
are $E$-stable equivalences.
Then, for any object $M$ of $\mathcal{C}$, we have the commutative diagram
with exact columns
\[\xymatrix{
\Ext^1_{\mathcal{C}/E}(X^{\prime}, M)\ar[d]\ar[r]^{\cong} &  \Ext^1_{\mathcal{C}/E}(X, M)\ar[d] \\
\Ext^2_{\mathcal{C}/E}(Y^{\prime}/X^{\prime}, M)\ar[d]\ar[r]^{\cong} & \Ext^2_{\mathcal{C}/E}(Y/X, M)\ar[d] \\
\Ext^2_{\mathcal{C}/E}(Y^{\prime}, M)\ar[d]\ar[r] & \Ext^2_{\mathcal{C}/E}(Y, M)\ar[d] \\
\Ext^2_{\mathcal{C}/E}(X^{\prime}, M)\ar[d]\ar[r]^{\cong} & \Ext^2_{\mathcal{C}/E}(X, M)\ar[d] \\
\Ext^3_{\mathcal{C}/E}(Y^{\prime}/X^{\prime}, M)\ar[r]^{\cong} & \Ext^3_{\mathcal{C}/E}(Y/X, M) .}\]
The horizontal maps marked as isomorphisms are isomorphisms because
an $E$-stable equivalence $A\rightarrow B$ induces a natural equivalence
of functors $\Ext^i_{\mathcal{C}/E}(B, -)\stackrel{\cong}{\longrightarrow}
\Ext^i_{\mathcal{C}/E}(A, -)$ for all $i\geq 1$; this is 
Lemma~3.6 of \cite{rectifications}. By the Five Lemma, we then
have a natural isomorphism of functors
$\Ext^2_{\mathcal{C}/E}(Y^{\prime}, -)\stackrel{\cong}{\longrightarrow}
\Ext^2_{\mathcal{C}/E}(Y, -)$.
But since every $E$-projective is $E$-injective, this natural
transformation being a natural isomorphism implies that
$Y\rightarrow Y^{\prime}$ is an $E$-stable equivalence; this is
Lemma~4.2 of \cite{rectifications}. Hence the extension axiom is satisfied if $E=F$.
Now if we do not have $E=F$ but instead $F\subseteq E$, then the extension axiom remains satisfied, as we have
fewer diagrams to check the extension axiom for.
\end{itemize}
\end{proof}

\section{Absolute and relative quasi-Frobenius conditions.}

\subsection{Definitions.}

To our knowledge these definitions are all new. They are variants
of the condition that every object embeds in a projective object, which
Faith and Walker showed (see Theorem~\ref{faith-walker thm}) 
to be equivalent, for categories of 
modules over a ring, to the ring being quasi-Frobenius.

\begin{definition} \label{relative qf conditions}
\begin{itemize}
\item 
Let $\mathcal{C}$ be an abelian category, $E,F$ a pair of allowable classes in $\mathcal{C}$. We say that $\mathcal{C}$ is {\em cone-Frobenius relative to $E,F$}
if, for any object $X$ of $\mathcal{C}$, there exists an $F$-monomorphism from $X$ to an $E$-projective object
of $\mathcal{C}$. 

We say that $\mathcal{C}$ is {\em functorially cone-Frobenius relative to $E,F$} if there exists a functor $J: \mathcal{C}\rightarrow\mathcal{C}$
and a natural transformation $\eta: \id_{\mathcal{C}}\rightarrow J$ such that:
\begin{enumerate}
  \item $J(X)$ is $E$-projective for every object $X$ of $\mathcal{C}$,
  \item $\eta(X): X\rightarrow J(X)$ is an $F$-monomorphism for every object $X$ of $\mathcal{C}$, and
  \item if $f: X\rightarrow Y$ is an $F$-monomorphism then so is the map $J(f): J(X)\rightarrow J(Y)$, and so is the universal map
\begin{equation}\label{monotone growth} Y\coprod_X J(X) \rightarrow J(Y).\end{equation}
\end{enumerate}
We sometimes call the pair $J,\eta$ a {\em cone functor on $\mathcal{C}$ relative 
to $E,F$.} When $E,F$ are understood from context we simply call 
$J,\eta$ a {\em relative cone functor.}
\item {\bf (The absolute case.)}
If $E=F$ is the class of all short exact sequences in $\mathcal{C}$ and $\mathcal{C}$ is cone-Frobenius relative to $E,F$, 
then we simply say that $\mathcal{C}$ is {\em cone-Frobenius.}

If $E=F$ is the class of all short exact sequences in $\mathcal{C}$ and $\mathcal{C}$ is functorially cone-Frobenius relative to $E,F$, 
then we simply say that $\mathcal{C}$ is {\em functorially cone-Frobenius.}
\end{itemize}
\end{definition}
The idea behind this definition is that the 
$F$-monomorphism and $E$-projective object
together are a kind of ``mapping cone,'' in the sense of homotopy theory, on $X$:
an embedding into a contractible object. 

These definitions are very reasonable and in fact are very generally satisfied in the presence of much milder-looking
quasi-Frobenius conditions on the abelian category. We present some theorems to that effect from our paper \cite{injectiveresolutions}, using the following
definition:
\begin{definition}\label{def of monic gen set}
Let $\mathcal{C}$ be an abelian category, and let $E$ be an allowable class in $\mathcal{C}$.
We will say that a set of objects $\{ X_s\}_{s\in S}$ of $\mathcal{C}$ is 
an {\em $E$-monic generating set for $\mathcal{C}$} if the following condition is satisfied:
\begin{itemize}
\item If $f: M \rightarrow N$ is a morphism in $\mathcal{C}$ such that $f\circ g \neq 0$ for every $E$-monomorphism $g: X_s\rightarrow M$,
then $f$ is $E$-monic.\end{itemize}

If an $E$-monic generating set for $\mathcal{C}$ exists, then we say that $\mathcal{C}$ is {\em $E$-monically small}.
If $\kappa$ is a cardinal number and an $E$-monic generating set for $\mathcal{C}$ exists with $\leq \kappa$ objects in it, then we say that $\mathcal{C}$ is 
{\em $E$-monically $\kappa$-small}.
\end{definition}
For example, we prove the following in~\cite{injectiveresolutions}:
\begin{prop}\label{injectiveresolutions cited result 1}
Let $R$ be a ring and let $\Mod(R)$ be the category of left $R$-modules. Let $E$ be the absolute allowable class on $\mathcal{C}$, i.e.,
$E$ is the class of {\em all} short exact sequences in $\mathcal{C}$. Then $\Mod(R)$ is $E$-monically $\kappa$-small, where $\kappa$ is the number
of left ideals of $R$.
\end{prop}
Then we have the theorem, also from~\cite{injectiveresolutions}:
\begin{theorem}\label{injectiveresolutions cited result 2}
Suppose $\mathcal{C}$ is an abelian category, $E,F$ allowable classes in $\mathcal{C}$. Suppose $F$ has retractile monics, and suppose that
$\mathcal{C}$ is $F$-monically $\kappa$-small and $\mathcal{C}$ has all coproducts of cardinality $\leq \kappa$. 
Suppose that every object of $\mathcal{C}$ embeds 
in an $F$-injective object by an $F$-monomorphism, i.e., for every object $X$ there exists an $F$-monomorphism from $X$ to an $F$-injective object.
Suppose every $F$-injective object of $\mathcal{C}$ is $E$-projective.
Then $\mathcal{C}$ is functorially cone-Frobenius relative to $E,F$.
\end{theorem}
As a special case:
\begin{corollary}\label{injectiveresolutions cited result 3}
Suppose $\mathcal{C}$ is an abelian category with enough injectives, with all small coproducts, and which is monically small. Suppose that every 
injective object in $\mathcal{C}$ is projective.
Then $\mathcal{C}$ is functorially cone-Frobenius.
\end{corollary}
As an even more special case:
\begin{corollary}\label{injectiveresolutions cited result 4}
Let $R$ be a quasi-Frobenius ring. Then the category of $R$-modules is functorially cone-Frobenius.
\end{corollary}
Finally, there is the question of when we have the functorial cone-Frobenius condition on the {\em finitely generated} modules over a ring. This is another
result from~\cite{injectiveresolutions}:
\begin{corollary}\label{injectiveresolutions cited result 5}
Let $k$ be a finite field and let $R$ be an quasi-Frobenius 
associative $k$-algebra which is 
finite-dimensional as a $k$-vector space. Then the category of
finitely generated (left) $R$-modules is functorially cone-Frobenius.
\end{corollary}
So the functorial cone-Frobenius condition seems to be a reasonable and often-satisfied one.

We also want to consider a relative form of the cone-Frobenius condition:
\begin{definition}
\begin{itemize}
\item
If $\mathcal{C},\mathcal{C}^{\prime}$ are abelian categories and $i: \mathcal{C}^{\prime}\rightarrow \mathcal{C}$ is an additive functor with right adjoint $r$,
let $E$ denote the allowable class of all short exact sequences in $\mathcal{C}$ and let $E^{\prime}$ the allowable class of 
all short exact sequences in $\mathcal{C}^{\prime}$. Then we say that $\mathcal{C}^{\prime}$ is {\em relatively quasi-Frobenius over $\mathcal{C}$}
(resp. {\em functorially relatively quasi-Frobenius over $\mathcal{C}$})
if $\mathcal{C}^{\prime}$ is cone-Frobenius relative to $i^*E,r_*E$ (resp. functorially cone-Frobenius relative to $i^*E,r_*E$).
\end{itemize}
\end{definition}
The idea here is that, in $\mathcal{C}^{\prime}$, 
any object embeds into some other object
in such a way that, once one applies $i$, one gets a mapping cone 
(an $E$-monomorphic embedding into a contractible, i.e., $E$-projective, object)
in $\mathcal{C}$. If the reader is wondering whether such a thing
is really more than a cone functor on $\mathcal{C}^{\prime}$ or a
cone functor on $\mathcal{C}$, the reader might try letting
$\mathcal{C}$ and $\mathcal{C}^{\prime}$ be module categories over
rings, and let $i$ be the base-change/extension-of-scalars 
functor induced by a {\em surjection}
of rings. 
The fact that ring surjections are rarely flat and hence that $i$ is not generally
left exact, i.e., $i$ does not generally preserve monomorphisms, 
means that most attempts one might make to produce a mapping cone
on $\mathcal{C}^{\prime}$ do not give mapping cones on $\mathcal{C}$
after applying $i$. So the relative Frobenius conditions really are expressing
something nontrivial.

Now the Faith-Walker Theorem, stated above as Theorem~\ref{faith-walker thm},
can be restated using our definitions:
{\em a ring $R$ is quasi-Frobenius if and only if
the category of $R$-modules is cone-Frobenius.} 

\subsection{Existence of cylinder functors.}

Now we show the equivalence of the functorial cone-Frobenius condition with 
the existence of cylinder functors satisfying the cylinder axioms.

First, we need a couple of lemmas:
\begin{lemma}\label{direct sum of elements of E is in E}
Suppose $\mathcal{C}$ is an abelian category, $E$ an allowable class in $\mathcal{C}$ with sectile epics.
Suppose $\mathcal{C}$ has enough $E$-projectives. Then any finite direct sum of members of $E$ is in $E$.
\end{lemma}
\begin{proof}
Let $I$ be a finite set and 
\[ 0\rightarrow X_i\rightarrow Y_i\rightarrow Z_i \rightarrow 0\]
be a member of $E$ for every $i\in I$. Then, for any $E$-projective object $P$ of $\mathcal{C}$,
we have the commutative diagram
\[\xymatrix{ 
\hom_{\mathcal{C}}(P, \oplus_i Y_i) \ar[r] \ar[d]^{\cong} & \hom_{\mathcal{C}}(P, \oplus_i Z_i) \ar[d]^{\cong} \\
\oplus_i \hom_{\mathcal{C}}(P, Y_i) \ar[r] & \oplus_i \hom_{\mathcal{C}}(P, Z_i).
}\]
The bottom horizontal map is a surjection of abelian groups, so the top horizontal map is as well. Now by Heller's theorem~\ref{heller's theorem},
the map $\oplus_i Y_i\rightarrow\oplus_i Z_i$ is an $E$-epimorphism.
So the short exact sequence
\[ 0 \rightarrow \oplus_i X_i \rightarrow \oplus_i Y_i\rightarrow \oplus_i Z_i \rightarrow 0\]
is in $E$.
\end{proof}

\begin{lemma} {\bf (Shearing $E$-monics.)} \label{shearing E-monics}
Let $\mathcal{C}$ be an abelian category and let $E$ be an allowable class in $\mathcal{C}$.
Suppose $X,Y,Z$ are objects in $\mathcal{C}$ and suppose we have $E$-monomorphisms
$e: X\rightarrow Y$ and $f: Z\rightarrow Y$.
Let $s$ be the morphism 
\[ s: X\oplus Z\rightarrow Y\oplus Z\]
given by the matrix of maps
\[ s= \left[ \begin{array}{ll} e & f \\ 0 & \id_Z \end{array}\right] .\]
Then $\coker s$ is naturally isomorphic to $\coker e$.
Furthermore, if $\mathcal{C}$ has enough $E$-injectives and $E$ has retractile monics, then $s$ is an $E$-monomorphism.
\end{lemma}
\begin{proof}
We first show that $\coker e\cong \coker s$. But this follows immediately from the commutative diagram with exact rows and exact columns:
\[\xymatrix{
0 \ar[r]\ar[d] & 0 \ar[r]\ar[d] & 0 \ar[r]\ar[d] & 0 \ar[r]\ar[d] & 0 \ar[d] \\
0 \ar[r] \ar[d] & X \ar[r]^{e} \ar[d]^i & Y \ar[r]\ar[d]^i & \coker e \ar[r]\ar[d] & 0 \ar[d] \\
0 \ar[r] \ar[d] & X\oplus Z \ar[r]^s \ar[d]^{\pi} & Y\oplus Z \ar[d]^{\pi} \ar[r] & \coker s \ar[d] \ar[r] & 0 \ar[d] \\
0 \ar[r] \ar[d] & Z \ar[r]^{\id} \ar[d] & Z \ar[d] \ar[r] & 0 \ar[d] \ar[r] & 0 \ar[d] \\
0\ar[r] & 0\ar[r] & 0\ar[r] & 0\ar[r] & 0
}\]
in which the maps marked $\pi$ are projections to the second summand, and the maps marked $i$ are inclusions as the first summand.

Now assume that $E$ has retractile monics, and let $t: Y\oplus Z \rightarrow Y\oplus Y$ be the map given by the matrix of maps
\[ t = \left[ \begin{array}{ll} \id_Y & -f \\ 0 & f \end{array} \right] .\]
Then a matrix multiplication reveals that the composite map
$t\circ s: X\oplus Z\rightarrow Y\oplus Y$ is the direct sum map $e\oplus f$, a direct sum of $E$-monomorphisms, hence by 
Lemma~\ref{direct sum of elements of E is in E}, itself an $E$-monomorphism. (Note that, by taking the opposite category and noticing that
the definition of an allowable class in an abelian category is self-dual, we get the conclusion of Lemma~\ref{direct sum of elements of E is in E}
if $E$ has retractile monics and $\mathcal{C}$ has enough $E$-injectives.)
Now since $t\circ s$ is an $E$-monomorphism and
$E$ is assumed to have retractile monics, $s$ is also an $E$-monomorphism.
\end{proof}

\begin{theorem} 
\label{existence of cylinder functors}
Let $\mathcal{C},E,F$ be as in Proposition-Definition~\ref{waldhausen cat from an allowable class}.
Then the Waldhausen category 
$\mathcal{C}_{E-we}^{F-cof}$ admits a cylinder functor satisfying the cylinder
axiom if and only if $\mathcal{C}$
is functorially cone-Frobenius relative to $E,F$.
\end{theorem}
\begin{proof}
If $\mathcal{C}_{E-we}^{F-cof}$ has a cylinder functor satisfying the cylinder
axiom, then
the cylinder functor $I$ sends, for any object $X$ of $\mathcal{C}$, the map 
$X\rightarrow \pt$ to the diagram
\[ \xymatrix{ X \ar[r]^{F-cof}\ar[rd] & I(X)\ar[d]^{E-we} & \pt\ar[l]\ar[ld] \\
 & \pt & }\]
where the map marked $F-cof$ is an $F$-monic and the map marked $E-we$ is
an $E$-stable equivalence. But for an object's map to the zero object
to be an $E$-stable equivalence, this is equivalent to that object
being an $E$-projective. Furthermore, if $f:X\rightarrow Y$
is any $F$-monomorphism in $\mathcal{C}$,
then we have the commutative diagram
\[ \xymatrix{
 X \coprod \pt \ar[d]^{f\coprod \id_{\pt}} \ar[r] & 
  I(X\coprod \pt) \ar[d]^{I(f\coprod \id_{\pt})} \\
 Y \coprod \pt \ar[r] &
  I(Y\coprod \pt) }\]
and condition (Cyl 1) in the definition of a cylinder
functor requires that $I(f \coprod \id_{\pt})$
be an $F$-monomorphism as well.
This completes one direction of the proof: $X \mapsto I(X\rightarrow \pt)$
is a relative cone functor.

Now suppose we have a functor $J$ and natural transformation $e$ as in the definition
of a relative cone functor in Definition~\ref{relative qf conditions}.
We claim that the functor sending any map $X\stackrel{f}{\longrightarrow} Y$ to the diagram
\[\xymatrix{ 
X\ar[r]^{(f,e(X))} \ar[rd]_f & 
 Y\oplus J(X) \ar[d]^{\pi_Y} & 
 Y\ar[l]_{(\id,0)} \ar[ld]^{\id} \\
 & Y & , }\]
where $\pi_Y$ is projection to $Y$, is a cylinder functor on $\mathcal{C}$
satisfying the cylinder axiom. We check Waldhausen's conditions from
Definition~\ref{def of waldhausen cat}. 
Condition (Cyl 2) is immediate, and the cylinder axiom 
follows from the projection 
$J(X)\oplus Y \rightarrow \pt\oplus Y\cong Y$
being a direct sum of $E$-stable equivalences, hence itself
an $E$-stable equivalence.

We handle condition (Cyl 1) as follows: 
if $X^{\prime}\rightarrow X$ and $Y^{\prime}\rightarrow Y$
are $F$-monomorphisms, then the 
direct sum 
$J(X^{\prime})\oplus Y^{\prime}\rightarrow J(X)\oplus Y$
is an $F$-monomorphism by the assumption that $J$ sends
$F$-monomorphisms to $F$-monomorphisms, and 
\[ X \coprod_{X^{\prime}} J(f^{\prime}) \coprod_{Y^{\prime}} Y \rightarrow J(f)\]
being an $F$-monomorphism is exactly equivalent to the 
map~\ref{monotone growth} being an $F$-monomorphism, which is part of the definition of the functorial cone-Frobenius condition.

Now suppose that $X^{\prime}\rightarrow X$ and $Y^{\prime}\rightarrow Y$
are $E$-stable equivalences.
We also note that, since $J(X^{\prime}),J(X)$ are
$E$-projective, the projections $J(X^{\prime})\rightarrow \pt$
and $J(X)\rightarrow \pt$ are $E$-stable equivalences,
and Lemma~\ref{composite of stable equivalences is a stable equivalence} gives us that the composite of either one with
an $E$-stable inverse of the other is an $E$-stable
equivalence between $J(X)$ and $J(X^{\prime})$. So
$J(X^{\prime})\oplus Y^{\prime}\rightarrow J(X)\oplus Y$
is an $E$-stable equivalence. 

So condition (Cyl 1) holds.

Lemma~\ref{shearing E-monics} implies that the direct sum map
\[ \left[\vcenter{\xymatrix{ \id_Y & f \\ 0 & e(X) } }\right] :
Y \oplus X \rightarrow Y\oplus J(X)\]
is an $E$-monomorphism, i.e., a cofibration, so condition (Cyl 3) holds.

On the other hand, one easily verifies that, if $\mathcal{C}$ has a cylinder 
functor,
then that cylinder functor is precisely the functor necessary for $\mathcal{C}$
to be functorially quasi-Frobenius category relative to $E,F$, from Definition~\ref{relative qf conditions}.
\end{proof}

\begin{corollary}\label{main cor}
Suppose $\mathcal{C}$ is a quasi-Frobenius abelian category with 
enough projectives and functorially
enough injectives. Then $\mathcal{C}$ admits the structure of a Waldhausen
category in which the cofibrations are the monomorphisms and the
weak equivalences are the stable equivalences. Furthermore,
this Waldhausen category satisfies the 
saturation and extension axioms, and it admits a cylinder functor
satisfying the cylinder axiom.
We call this Waldhausen category the {\em stable $G$-theory of $\mathcal{C}$},
and we call its associated $K$-groups the {\em stable $G$-theory
groups of $\mathcal{C}$.}
\end{corollary}

\begin{corollary}\label{algebraic main cor}
Suppose $k$ is a finite field and $R$ a quasi-Frobenius $k$-algebra which is 
finite-dimensional as a $k$-vector space.
Then the category of finitely generated (left) $R$-modules
admits the structure of a Waldhausen
category in which the cofibrations are the monomorphisms and the
weak equivalences are the stable equivalences. Furthermore,
this Waldhausen category satisfies the 
saturation and extension axioms, and it admits a cylinder functor
satisfying the cylinder axiom.
\end{corollary}

\subsection{Multiplicative structure.}

We recall Waldhausen's construction of multiplicative structures on the $K$-theory of Waldhausen categories:
\begin{theorem} {\bf (Waldhausen.)} \label{waldhausen product thm}
Suppose $\mathcal{A},\mathcal{B},\mathcal{C}$ are Waldhausen categories 
and $F: \mathcal{A}\times \mathcal{B} \rightarrow \mathcal{C}$ is a functor satisfying each of the following conditions:
\begin{itemize}
\item\label{wh product condition 1} If $M \rightarrow N$ is a cofibration in $\mathcal{A}$ and $L$ is an object of $\mathcal{B}$, then
$F(M, L) \rightarrow F(N, L)$ is a cofibration in $\mathcal{C}$.
\item\label{wh product condition 2} If $M \rightarrow N$ is a cofibration in $\mathcal{B}$ and $L$ is an object of $\mathcal{C}$, then
$F(L, M) \rightarrow F(L, N)$ is a cofibration in $\mathcal{C}$.
\item\label{wh product condition 3} If $M^{\prime}\rightarrow M$ and $N^{\prime}\rightarrow N$ are cofibrations in $\mathcal{A}$ and $\mathcal{B}$, respectively,
then the universal map 
\[ F(M^{\prime}, N)\coprod_{F(M^{\prime},N^{\prime})} F(M, N^{\prime}) \rightarrow F(M,N)\]
is a cofibration in $\mathcal{C}$.
\end{itemize}
Then $F$ induces a natural pairing 
\[ \Omega \left| wS_{\cdot} \mathcal{A}\right| \times \Omega \left| wS_{\cdot} \mathcal{B}\right| \rightarrow
 \Omega \Omega \left| wS_{\cdot}S_{\cdot} \mathcal{C}\right|\stackrel{\cong}{\longrightarrow} \Omega\left| wS_{\cdot} \mathcal{C}\right|. \]
\end{theorem}

Waldhausen's pairing is sufficiently natural to imply the following (see~\cite{MR2544391} for a good expository account of infinite loop spaces with
ring structure, but note that the statement we can make, below, is about {\em homotopy-commutative} ring spaces and spectra, but {\em not}
$E_\infty$-ring spaces or spectra):
\begin{corollary}\label{wh product cor}
Suppose $\mathcal{C}$ is a Waldhausen category which is equipped with a symmetric monoidal product
$X, Y \mapsto X\otimes Y$ satisfying the conditions for the functor $F$ in Theorem~\ref{waldhausen product thm}.
Then $\Omega \left| wS_{\cdot} \mathcal{C}\right|$ is a homotopy-commutative ring space, that is,
the infinite loop space $\Omega \left| wS_{\cdot} \mathcal{C}\right|$ is equipped with a homotopy-commutative product that
is compatible with its loop space structure. 

Equivalently, when regarded as a spectrum using its infinite loop space structure,
$\Omega \left| wS_{\cdot} \mathcal{C}\right|$ is a homotopy-commutative ring spectrum.
\end{corollary}

\begin{lemma}\label{elementary linear algebra lemma}
Let $k$ be a field and let $M^{\prime},M,N^{\prime},N$ be finite-dimensional
$k$-vector spaces. Suppose we have monomorphisms
$M^{\prime}\rightarrow M$ and $N^{\prime}\rightarrow N$ of $k$-vector spaces.
Then the canonical map
\[ M^{\prime}\otimes_k N\coprod_{M^{\prime}\otimes_k N^{\prime}} M\otimes_k N^{\prime} \rightarrow M\otimes_k N\]
is a monomorphism.
\end{lemma}
\begin{proof}
We will regard $M^{\prime}$ and $N^{\prime}$ as subspaces of $M$ and $N$, respectively.
Choose a $k$-linear basis $\{ m_1, \dots , m_i\}$ for $M$ and a $k$-linear
basis $\{ n_1, \dots ,n_j\}$ for $N$ such that their restrictions to initial
subsequences $\{ m_1, \dots ,m_{i^{\prime}}\}$ and $\{ n_1, \dots ,n_{j^{\prime}}\}$
form $k$-linear bases for $M^{\prime}$ and $N^{\prime}$, respectively.
Then we have the short exact sequence of $k$-vector spaces
\[ 0 \rightarrow M^{\prime}\otimes_k N^{\prime}
 \stackrel{f}{\longrightarrow} M^{\prime}\otimes_k N\oplus M\otimes_k N^{\prime}
 \rightarrow M^{\prime}\otimes_k N\coprod_{M^{\prime}\otimes_k N^{\prime}} M\otimes_k N^{\prime}
 \rightarrow 0.\]
We write $p$ for the composite map
\begin{equation}\label{composite map 1} M^{\prime}\otimes_k N\oplus M\otimes_k N^{\prime}\rightarrow M^{\prime}\otimes_k N\coprod_{M^{\prime}\otimes_k N^{\prime}} M\otimes_k N^{\prime} \rightarrow M\otimes_k N.\end{equation}
Suppose we have an element 
\[ x = \left( \sum_{a=1}^{i^{\prime}}\sum_{b=1}^j \alpha_{a,b} m_a\otimes n_b ,
     \sum_{a=1}^{i}\sum_{b=1}^{j^{\prime}} \beta_{a,b} m_a\otimes n_b \right)
  \in  M^{\prime}\otimes_k N\oplus M\otimes_k N^{\prime} \]
such that $p(x) = 0$. 
Then we have:
\begin{eqnarray*}
0 & = & \sum_{a=1}^{i^{\prime}}\sum_{b=1}^j \alpha_{a,b} m_a\otimes n_b - 
 \sum_{a=1}^{i}\sum_{b=1}^{j^{\prime}} \beta_{a,b} m_a\otimes n_b  \\
  & = & \sum_{a=1}^{i^{\prime}}\sum_{b=1}^{j^{\prime}} (\alpha_{a,b} - \beta_{a,b}) m_a\otimes n_b  + \sum_{a=1}^{i^{\prime}}\sum_{b=j^{\prime}+1}^{j} \alpha_{a,b} m_a\otimes n_b
 + \sum_{a=i^{\prime}+1}^{i}\sum_{b=1}^{j^{\prime}} -\beta_{a,b} m_a\otimes n_b,
 \end{eqnarray*}
hence $\beta_{a,b} = 0$ for all $a>i^{\prime}$, and $\alpha_{a,b} = 0$ for all
$b>j^{\prime}$, and $\alpha_{a,b} = \beta_{a,b}$ if $a \leq i^{\prime}$
and $b \leq j^{\prime}$. These conditions together imply that
$x$ is in the subspace $M^{\prime}\otimes_k N^{\prime}$ of
$M^{\prime}\otimes_k N\oplus M\otimes_k N^{\prime}$, hence the kernel of the
composite map~\ref{composite map 1} is contained in 
$M^{\prime}\otimes_k N^{\prime}$, hence that the map from the quotient
\[ \left( M^{\prime}\otimes_k N\oplus M\otimes_k N^{\prime}\right)/\left(M^{\prime}\otimes_k N^{\prime}\right) \cong M^{\prime}\otimes_k N\coprod_{M^{\prime}\otimes_k N^{\prime}} M\otimes_k N^{\prime}\]
to $M\otimes_k N$ is injective. 
\end{proof}

\begin{prop}\label{existence of multiplicative structure on stable g-thy}
Let $k$ be a field and let $A$ be a co-commutative Hopf algebra over $k$ which is finite-dimensional as a $k$-vector space.
Let $\mathcal{G}_{st}(A)$ be the stable $G$-theory Waldhausen category of $A$
from Corollary~\ref{main cor}, that is, $\mathcal{G}_{st}(A)$
is the category of finitely generated (left) $A$-modules, with
cofibrations the monomorphisms and weak equivalences the stable equivalences.
Then the tensor product $\otimes_k$ over $k$ 
gives $\mathcal{G}_{st}(A)$ a symmetric monoidal product satisfying the assumptions
of Corollary~\ref{wh product cor}.
Hence the stable $G$-theory spectrum
$\Omega \left| wS_{\cdot} \mathcal{C}\right|$ has the structure of a homotopy-commutative ring spectrum.
\end{prop}
\begin{proof}
Finite-dimensional Hopf algebras over fields are known to be quasi-Frobenius; 
see e.g.~\cite{MR738973} for this fact.
That the tensor product over $k$ is symmetric monoidal follows from
the Hopf algebra being co-commutative. That the tensor product over $k$
satisfies the first two conditions
of Definition~\ref{waldhausen product thm}
is because every $A$-module is flat as a $k$-module, since $k$ is a field,
so the tensor product over $k$ preserves monomorphisms.
That the tensor product over $k$ satisfies the third condition
is because of Lemma~\ref{elementary linear algebra lemma}.
\end{proof}

\section{Applications.}

\subsection{The relationship between stable $G$-theory and other $G$-theories and $K$-theories.}

We recall that stable $G$-theory was defined in 
Corollary~\ref{main cor}. It sits naturally in a diagram relating
it to algebraic $K$-theory, algebraic $G$-theory, and the
``derived representation groups,'' but the relationship between these
theories does not seem to be 
simple enough to permit easy computation of one from the others. 
We describe this relationship and provide a few comments in
Remark~\ref{remark on computations} about the
computational task of computing one of these theories once one has computed
the others.

Now we define some notations we use to describe certain Waldhausen categories associated to an abelian
category $\mathcal{C}$:
\begin{definition}
Suppose $\mathcal{C}$ is an abelian category.
We will write:
\begin{itemize}
\item 
$\mathcal{K}^{\oplus}(\mathcal{C})$ for the
{\em split $K$-theory category of $\mathcal{C}$}, i.e., the Waldhausen category
structure on the full subcategory generated by the projective objects of 
$\mathcal{C}$, where cofibrations are split inclusions and
weak equivalences are isomorphisms.
\item 
$\mathcal{K}(\mathcal{C})$ for the
{\em nonsplit $K$-theory category of $\mathcal{C}$}, i.e., the Waldhausen category
structure on the full subcategory generated by the projective objects of 
$\mathcal{C}$, where cofibrations are inclusions and
weak equivalences are isomorphisms.
\item 
$\mathcal{G}^{\oplus}(\mathcal{C})$ for the
{\em split $G$-theory category of $\mathcal{C}$}, i.e., the Waldhausen category
structure on $\mathcal{C}$ where cofibrations are split inclusions and
weak equivalences are isomorphisms.
\item 
$\mathcal{G}(\mathcal{C})$ for the
{\em $G$-theory category of $\mathcal{C}$}, i.e., the Waldhausen category
structure on $\mathcal{C}$ where cofibrations are inclusions and
weak equivalences are isomorphisms.
\item When it exists:
$\mathcal{G}^{\oplus}_{st}(\mathcal{C})$ for the
{\em stable split $G$-theory category of $\mathcal{C}$}, i.e., the Waldhausen category
structure on $\mathcal{C}$ where cofibrations are split inclusions and
weak equivalences are stable equivalences. (In the absolute sense,
i.e., $E$-stable equivalences where $E$ is the allowable class of
all short exact sequences in $\mathcal{C}$.)
\item When it exists:
$\mathcal{G}_{st}(\mathcal{C})$ for the
{\em stable $G$-theory category of $\mathcal{C}$}, i.e., the Waldhausen category
structure on $\mathcal{C}$ where cofibrations are inclusions and
weak equivalences are stable equivalences. (In the absolute sense,
i.e., $E$-stable equivalences where $E$ is the allowable class of
all short exact sequences in $\mathcal{C}$.)
\end{itemize}
\end{definition}

\begin{remark}\label{remark on K_0}
We note that, if $R$ is a ring and $\mathcal{C}$ is the category of finitely generated $R$-modules,
then 
\[ \pi_n\Omega \left| wS_{\cdot}\mathcal{K}^{\oplus}(\mathcal{C})\right|
\cong K_n(R),\]
i.e., the Waldhausen $K$-theory of the split $K$-theory Waldhausen category
recovers the classical algebraic $K$-theory of $R$.
Meanwhile,
\[ \pi_n\Omega \left| wS_{\cdot}\mathcal{G}(\mathcal{C})\right|
\cong G_n(R),\]
i.e., the Waldhausen $G$-theory of the (nonsplit) $G$-theory Waldhausen category
recovers the classical algebraic $G$-theory of $R$.

The other theories are more obscure but still meaningful.
In degree zero, 
\[ \pi_0\Omega \left| wS_{\cdot}\mathcal{G}^{\oplus}(\mathcal{C})\right|
\cong \Rep(R),\]
the representation group (actually ring, under tensor product; but we have not discussed any multiplicative structures on our Waldhausen categories, which is another subject entirely) of $R$---that is, the Grothendieck group completion of the 
monoid of isomorphism classes of finitely-generated $R$-modules. So we sometimes regard the split $G$-theory
as the ``derived representation theory'' and the groups
$\pi_*\Omega \left| wS_{\cdot}\mathcal{G}^{\oplus}(\mathcal{C})\right|$
as the ``derived representation groups'' of $R$.

In degree zero, 
\[ \pi_0\Omega \left| wS_{\cdot}\mathcal{G}^{\oplus}_{st}(\mathcal{C})\right|
\cong \StableRep(R),\]
the stable representation group of $R$---that is, the Grothendieck group completion of the 
monoid of stable equivalence classes of finitely-generated $R$-modules. So we sometimes regard the split stable $G$-theory
as the ``derived stable representation theory'' and the groups
$\pi_*\Omega \left| wS_{\cdot}\mathcal{G}^{\oplus}_{st}(\mathcal{C})\right|$
as the ``derived stable representation groups'' of $R$.

Finally, the results of this paper
are really about the stable $G$-theory groups
\[ \pi_*\Omega \left| wS_{\cdot}\mathcal{G}_{st}(\mathcal{C})\right|
\cong (G_{st})_*(\mathcal{C}) ,\]
as defined in Corollary~\ref{main cor}.
In degree zero, (nonsplit) stable $G$-theory is
\[ \pi_0\Omega \left| wS_{\cdot}\mathcal{G}_{st}(\mathcal{C})\right|
\cong \StableRep(R)/A,\]
the stable representation group modulo the subgroup $A$ generated
by all elements of the form $L - M + N$ where
\[ 0\rightarrow L \rightarrow M \rightarrow N \rightarrow 0\]
is a short exact sequence in $\mathcal{C}$.
\end{remark}

\begin{prop} \label{relation between g and k}
For any abelian category $\mathcal{C}$,
we have a commutative diagram of topological spaces
\begin{equation}\label{main homotopy fibration square} \xymatrix{ 
 \left| wS_{\cdot}\mathcal{K}^{\oplus}(\mathcal{C})\right| \ar[r]\ar[d] &
 \left| wS_{\cdot}\mathcal{G}^{\oplus}(\mathcal{C})\right| \ar[d]\ar[r] &
 \left| wS_{\cdot}\mathcal{G}^{\oplus}_{st}(\mathcal{C})\right| \ar[d] \\
 \left| wS_{\cdot}\mathcal{K}(\mathcal{C})\right| \ar[r] &
 \left| wS_{\cdot}\mathcal{G}(\mathcal{C})\right| \ar[r] & 
 \left| wS_{\cdot}\mathcal{G}_{st}(\mathcal{C})\right| 
 }\end{equation}
in which the horizontal composites are nulhomotopic.

Suppose further that every projective object in $\mathcal{C}$ is also injective.
Then the map
\[  \left| wS_{\cdot}\mathcal{K}^{\oplus}(\mathcal{C})\right| \rightarrow  \left| wS_{\cdot}\mathcal{K}(\mathcal{C})\right| \]
in the above diagram is a homotopy equivalence.

Finally, suppose $\mathcal{C}$ is functorially quasi-Frobenius, has enough projectives, has enough injectives, and every projective object is injective. 
Then the bottom row in diagram~\ref{main homotopy fibration square}
is a homotopy fiber sequence.
\end{prop}
\begin{proof}
First, when every projective object in $\mathcal{C}$ is also injective,
then any injective map between projective objects splits, by the universal
property of an injective object; so injections and split injections coincide
in the categories of projective objects in $\mathcal{C}$,
so the functor $\mathcal{K}^{\oplus}(\mathcal{C})\rightarrow
\mathcal{K}(\mathcal{C})$ is an isomorphism of Waldhausen categories.

That the bottom row, under the stated assumptions, 
is a homotopy fiber sequence is an immediate consequence
of Waldhausen's Fibration Theorem~\ref{fibration thm} and
our Theorem~\ref{existence of cylinder functors}.

Everything else here is a consequence of elementary facts from \cite{MR802796}.
\end{proof}

\begin{remark} \label{remark on computations}
There is a computational tool we would really like to have, but don't: we would very much like to be able to compute the homotopy groups of the fiber
$F^1$ of the map
\[ \left| wS_{\cdot} \mathcal{G}^{\oplus}_{st}(\mathcal{C})\right|\rightarrow\left| wS_{\cdot} \mathcal{G}_{st}(\mathcal{C})\right| .\]
From the associated long exact sequence induced by $\pi_*$ as well as the identification we have of $(G^{\oplus}_{st})_0$ and $(G_{st})_0$,
we know (for example) that $\pi_0$ surjects on to the subgroup of the stable representation group $\StabRep(\mathcal{C})$ generated by elements of the form
$M_0 - M_1+ M_2$ for all short exact sequences $0 \rightarrow M_0\rightarrow M_1\rightarrow M_2\rightarrow 0$ in $\mathcal{C}$, but it is hard to say much
about $\pi_i(F^1)$ for $i>0$. Since $(G_{st}^{\oplus})_0(\mathcal{C})$ is the stable representation group of $\mathcal{C}$, one should think of
$(G_{st}^{\oplus})_i(\mathcal{C})$ as the {\em derived} stable representation groups of $\mathcal{C}$. An identification of $\pi_*(F^1)$ would allow one to pass
from a knowledge of $(G_{st})_i(\mathcal{C})$, computable using our localization results in this paper, to a knowledge of the
derived stable representation groups $(G_{st}^{\oplus})_i(\mathcal{C})$. This would be valuable. 

The problem of identifying $F^1$ is a special case of the following general problem: when one changes (expands or contracts) the class of cofibrations in a Waldhausen category $\mathcal{C}$, how does it change the $K$-theory groups $\pi_*(\Omega \left| wS_{\cdot}\mathcal{C}\right| )$? We have an approach to this problem,
especially in the situation here of identifying $F^1$,
in which we use a relative cell decomposition to describe
the homology $H_n(\left| wS_{\cdot}\mathcal{C}\right|, \left| wS_{\cdot}\mathcal{C}^{\oplus}\right|)$ and then, making use of this homology computation and the relative Hurewicz map, we use an analogue of Serre's method of computing
homotopy groups of spheres to work our way up the relative Postnikov tower of
the map $\left| wS_{\cdot}\mathcal{C}^{\oplus}\right|\rightarrow\left| wS_{\cdot}\mathcal{C}\right|$. This technique is difficult and in any case is 
beyond the scope
of the present paper. We hope to return to it in a later paper.
\end{remark}

\subsection{Applications to algebras.}

In this section we will finally apply our results to actual rings and algebras! We will frequently assume that the rings in question are quasi-Frobenius.
For an explanation of why this implies the functorial cone-Frobenius condition, see
Proposition~\ref{injectiveresolutions cited result 1}, Theorem~\ref{injectiveresolutions cited result 2}, and Cors.~\ref{injectiveresolutions cited result 3},~\ref{injectiveresolutions cited result 4}, and~\ref{injectiveresolutions cited result 5}.
Throughout this section, whenever we assume that an algebra is over a {\em finite} field, the only reason we assume finiteness of the field is so that
the above construction gives a cone functor on the finitely-generated module category; if one can extend this construction to finitely-generated modules over
algebras over more general fields, then one can do away with the finiteness assumption on the field.

\begin{prop}{\bf (Stable $G$-theory is a delooping of relative algebraic $K$-theory.)}\label{stable G is delooping}
Suppose $k$ is a finite field and $A$ is a quasi-Frobenius 
$k$-algebra which is finite-dimensional as a $k$-vector space. Suppose
$f:A \rightarrow k$ is a surjective morphism of $k$-algebras with nilpotent kernel. 
Write $K(A\rightarrow k)$ for the relative $K$-theory space of
$f$, that is, the fiber of the map
\[ \Omega \left| wS_{\cdot}\mathcal{K}(A)\right| \rightarrow\Omega \left| wS_{\cdot}\mathcal{K}(k)\right| .\]
Then we have a homotopy equivalence $K(A \rightarrow k) \simeq \Omega^2\left|\mathcal{G}_{st}(A)\right|$, hence
\[ (G_{st})_n(A) \cong K_{n-1}(A \rightarrow k)\]
for all $n>0$.
\end{prop}
\begin{proof}
We have the commutative diagram of (pointed) spaces
\begin{equation}\label{comm diag 100} \xymatrix{
K(A\rightarrow k) \ar[r]\ar[d] & 
\Omega \left| wS_{\cdot}\mathcal{K}(A)\right| \ar[r] \ar[d] & 
\Omega \left| wS_{\cdot}\mathcal{K}(k)\right|  \ar[d] \\
G(A\rightarrow k) \ar[r]\ar[d] & 
\Omega \left| wS_{\cdot}\mathcal{G}(A)\right| \ar[r] \ar[d] & 
\Omega \left| wS_{\cdot}\mathcal{G}(k)\right|  \ar[d] \\
\Omega \left| wS_{\cdot}\mathcal{G}_{st}(A)\right| \ar[r] & 
\Omega \left| wS_{\cdot}\mathcal{G}_{st}(A)\right| \ar[r]  & 
\pt }\end{equation}
in which the spaces $K(A\rightarrow k)$ and $G(A\rightarrow k)$ are defined
to be the homotopy fibers of the evident maps, so that the rows
in diagram~\ref{comm diag 100} are homotopy fiber sequences.

We note that the middle and right-hand columns in diagram~\ref{comm diag 100}
are also homotopy fiber sequences, by Proposition~\ref{relation between g and k}.
Since all rows and the rightmost two columns in diagram~\ref{comm diag 100}
are homotopy fiber sequences, so is the left-hand column.
By Quillen's filtration argument in \cite{MR0338129},
one also knows that any nilpotent extension of algebras over a field
induces a homotopy equivalence in $G$-theory spaces, so
the map
$\Omega \left| wS_{\cdot}\mathcal{G}(A)\right| \rightarrow
\Omega \left| wS_{\cdot}\mathcal{G}(k)\right|$
is a homotopy equivalence.
So its fiber
$G(A\rightarrow k)$ is contractible. 
Hence the left-hand column of diagram~\ref{comm diag 100}
reads, up to homotopy equivalence,
\[ K(A \rightarrow k) \rightarrow \pt \rightarrow \Omega \left| wS_{\cdot}\mathcal{G}_{st}(A)\right| .\]
Hence $K(A\rightarrow k) \simeq \Omega^2 \left| wS_{\cdot}\mathcal{G}_{st}(A)\right|$.
\end{proof}

Now we recall Gabber's rigidity theorem, from e.g. Chapter 4 of~\cite{MR3076731}:
\begin{theorem}{\bf (Gabber.)}\label{gabber rigidity}
Suppose $A$ is a commutative ring and $(A,I)$ is a Hensel pair.
Suppose $1/i\in A$. Then, for all $n>0$, the map
\[ K_n(A; \mathbb{Z}/i\mathbb{Z}) \rightarrow 
K_n(A/I; \mathbb{Z}/i\mathbb{Z})\]
is an isomorphism.
\end{theorem}

\begin{theorem}\label{main thm computing Gst}
Suppose $k$ is a finite field of characteristic $p$ 
and $A$ is a quasi-Frobenius commutative
$k$-algebra which is finite-dimensional as a $k$-vector space. Suppose
$f:A \rightarrow k$ is a surjective morphism of $k$-algebras with nilpotent kernel.
Then 
we have the following computation of the stable $G$-theory groups of 
$A$:
\begin{itemize}
\item $(G_{st})_0(A)$ is the abelian group whose objects are the stable equivalence
classes of finitely-generated $A$-modules, modulo the relation $[L]+[N]=[M]$ if
there exists a short exact sequence $0\rightarrow L \rightarrow M \rightarrow N \rightarrow 0$.
\item For $n>0$, $(G_{st})_n(A)$ is uniquely $\ell$-divisible for all primes
$\ell\neq p$. (So $(G_{st})_n(A)$ has no summands of the form $\mathbb{Z}/\ell^i\mathbb{Z}$ or $\hat{\mathbb{Z}}_{\ell}$ or $(\mathbb{Q}/\mathbb{Z})_{\ell}$,
but it may have summands of the form $\mathbb{Q}$, for example.)
\item For $n>0$, the modulo $p^i$ stable $G$-theory 
$(G_{st})_n(A; \mathbb{Z}/p^i\mathbb{Z})$ is isomorphic to the topological cyclic homology
$TC_{n-1}(A; \mathbb{Z}/p^i\mathbb{Z})$.
\end{itemize}
\end{theorem}
\begin{proof}
\begin{itemize}
\item The statement about $(G_{st})_0(A)$ is the usual identification of $K_0$
of a Waldhausen category; see Remark~\ref{remark on K_0}.
\item Suppose $\ell$ is a prime distinct from $p$.
From Gabber rigidity, Theorem~\ref{gabber rigidity}, 
we know that $K_n(A\rightarrow k; \mathbb{Z}/\ell^i\mathbb{Z}) \cong 0$ for 
$n\geq 1$. Since $A$ is local, we also know that
the map $K_0(A) \rightarrow K_0(k)$ is an isomorphism (the identity map
from $\mathbb{Z}$ to $\mathbb{Z}$), so 
$K_0(A\rightarrow k;\mathbb{Z}/\ell^i\mathbb{Z})$ vanishes as well.
From the short exact sequence
\[ 0 \rightarrow K_n(A \rightarrow k)/\ell^i \rightarrow
 K_n(A \rightarrow k; \mathbb{Z}/\ell^i\mathbb{Z}) \rightarrow
 \ell^i\mbox{\ torsion\ in\ }K_{n-1}(A \rightarrow k)\rightarrow 0\]
and vanishing of the middle term for $n>0$,
we then know that $K_n(A\rightarrow k)/\ell^i\cong 0$ for $n>0$,
and $K_n(A\rightarrow k)$ is $\ell^i$-torsion-free for $n\geq 0$.
So, in particular, $K_n(A\rightarrow k)$ is uniquely $\ell$-divisible
for all $n\geq 0$.
By Proposition~\ref{stable G is delooping},
we know that $(G_{st})_n(A)\cong K_{n-1}(A \rightarrow k)$
for $n\geq 1$, so for $n\geq 1$ we conclude that
$(G_{st})_n(A)$ is uniquely $\ell$-divisible.
\item McCarthy's theorem (see e.g. Madsen's article~\cite{MR1474979} for a good expository account) implies that
$K_n(A\rightarrow k; \mathbb{Z}/p^i\mathbb{Z})\cong TC_n(A\rightarrow k; \mathbb{Z}/p^i\mathbb{Z}$). Now use Proposition~\ref{stable G is delooping}.
\end{itemize}
\end{proof}

One of our results (in Theorem~\ref{main thm, multiplicative version}, below) involves proving that certain stable $G$-theory spectra are complex oriented.
To prepare for that result, we provide the definitions of a complex orientation.
\begin{definition} Let $E$ be a homotopy-commutative ring spectrum. By a {\em complex orientation} on $E$ we mean a choice of element
$\chi\in \tilde{E}^2(BU(1))$ with the property that $\chi$ maps to the multiplicative unit element $1\in E^0(\pt)$ under the composite
\[ \tilde{E}^2(BU(1)) \rightarrow \tilde{E}^2(S^2) \stackrel{\cong}{\longrightarrow} \tilde{E}^0(S^0)\stackrel{\cong}{\longrightarrow} E^0(\pt),\]
where the left-hand map is the map induced in $E$-cohomology by inclusion of the $2$-skeleton $S^2$ of $BU(1)$, the classifying space
of complex line bundles.

If $E$ admits a complex orientation, we sometimes say that $E$ is {\em complex oriented.}
\end{definition}
Complex orientations are important for geometric purposes, since the element $\chi$ behaves essentially like a first Chern class, allowing one to carry out 
geometric arguments using $E$-cohomology that require characteristic classes for line bundles. Complex orientations also connect powerfully to number theory,
via formal group laws: a complex orientation on a homotopy-commutative ring spectrum $E$ gives rise to a one-dimensional commutative formal group law
on $\pi_*(E)$, and much of the homotopy theory of complex oriented ring spectra can be described completely in terms of the
moduli theory of one-dimensional formal groups. Adams's book \cite{MR1324104} is an excellent reference for this material.

\begin{definition}
Suppose $k$ is a finite field and $A$ is a quasi-Frobenius 
$k$-algebra which is finite-dimensional as a $k$-vector space.
We will write $g_{st}(A)$ for the {\em connective stable $G$-theory spectrum of $A$}, that is,
the connective cover of the spectrum $\Omega \left| wS_{\cdot}\mathcal{G}_{st}(A)\right|$.
We will write $(g_{st})_i(A)$ for the $i$th homotopy group $\pi_i(g_{st}(A))$.
\end{definition}
If $A$ is also a co-commutative Hopf algebra over $k$,
then by Proposition~\ref{existence of multiplicative structure on stable g-thy} and
the fact that connective covers of homotopy-commutative ring spectra are also homotopy-commutative ring spectra,
$g_{st}(A)$ is a homotopy-commutative ring spectrum.

A note for the reader who is not used to connective covers: the essential property of the connective cover $x$ of a spectrum $X$
is that $\pi_iX \cong \pi_ix$ for $i\geq 0$, and $\pi_ix\cong 0 $ for $i<0$.

Now one wants to use Proposition~\ref{existence of multiplicative structure on stable g-thy}
to get a multiplicative version of Theorem~\ref{main thm computing Gst}.
So we have the following, which also uses Hesselholt-Madsen's computation
of the topological cyclic homology of truncated polynomial rings
(see e.g. the survey article~\cite{MR1474979}):
\begin{theorem}\label{main thm, multiplicative version}
Suppose $k$ is a finite field of characteristic $p$.
Let $A\cong k[x]/x^{p^n}$ for some positive integer $n$.
Then we have the following results:
\begin{itemize}
\item $(g_{st})_0(A) \cong \mathbb{Z}/p^n\mathbb{Z}$.
\item For all positive integers $m$, $(g_{st})_{2m}(A) \cong TC_{2m-1}(A; \hat{\mathbb{Z}}_p) \cong 
\mathbb{W}_{mp^n-1}(k)/V_{p^n}\mathbb{W}_{m-1}(k)$,
where $\mathbb{W}_i(k)$ is the standard filtration quotient $i$ 
of the ring of big Witt vectors, that is,
\[ \mathbb{W}_i(k) = (1+Xk[[X]])^{\times}/(1+X^{m+1}k[[X]])^{\times} ,\]
and $V_j$ is the Verschiebung morphism sending a power series $f(X)$
to the power series $f(X^{j})$.
\item For all positive integers $m$, $(g_{st})_{2m-1}(A) \cong 0$.
\item $g_{st}(A)$ is a complex oriented ring spectrum.
\end{itemize}
\end{theorem}
\begin{proof}
Throughout, we implicitly use the fact that $(G_{st})_i(A)\cong (g_{st})_i(A)$ for $i\geq 0$.
\begin{itemize}
\item To get the isomorphism $(G_{st})_0(A) \cong \mathbb{Z}/p^n\mathbb{Z}$
we use the identification of degree zero stable $G$-theory,
in Remark~\ref{remark on K_0}, as the group of stable equivalence
classes of finitely generated $A$-modules, modulo the relation
splitting all short exact sequences.
Now $A$ is a principal ideal domain, hence its finitely-generated module category is Krull-Schmidt and every indecomposable finitely-generated module
is cyclic, and for every $i\leq p^n$, we have the short exact sequence
\[ 0 \rightarrow k \rightarrow k[x]/x^i \rightarrow k[x]/x^{i-1} \rightarrow 0\]
of $A$-modules, 
so we have the relation $[k] + [k[x]/x^{i-1}] = [k[x]/x^i]$
in $(G_{st})_0(A)$, so $[k]$ generates $(G_{st})_0(A)$.
Finally, we observe that $[k[x]/x^{p^n}] = p^n[k]$ is zero in stable
$G$-theory. So $(G_{st})_0(A)\cong \mathbb{Z}/p^n\mathbb{Z}$, generated
by the stable equivalence class of the $A$-module $k$.
\item Note that we have a description of the stable $G$-theory of
$A$ in Theorem~\ref{main thm computing Gst} up to uniquely
$\ell$-divisible summands. But, 
by Proposition~\ref{existence of multiplicative structure on stable g-thy},
the spectrum of the infinite loop space 
$\Omega \left| wS_{\cdot} \mathcal{G}_{st}(A) \right|$ is a
homotopy-commutative ring spectrum. So each homotopy group 
$\pi_i\Omega \left| wS_{\cdot} \mathcal{G}_{st}(A) \right|\cong (G_{st})_i(A)$
is a module over $\pi_0\Omega \left| wS_{\cdot} \mathcal{G}_{st}(A) \right|
\cong (G_{st})_0(A) \cong\mathbb{Z}/p^n\mathbb{Z}$.
The only uniquely $\ell$-divisible group admitting the structure of a 
$\mathbb{Z}/p^n\mathbb{Z}$-module is the trivial group.
So the uniquely divisible factors whose existence is not ruled out
by Theorem~\ref{main thm computing Gst} are indeed ruled out 
in this case, by the multiplicative structure.

The computation of $TC_*(A; \hat{\mathbb{Z}}_p)$ is due to 
Hesselholt and Madsen, as in \cite{MR1474979}.
\item
That $(G_{st})_*(A)$ vanishes in odd positive degrees is similar to the computation in even degrees.
Theorem~\ref{main thm computing Gst} guarantees that, up to uniquely $\ell$-divisible
summands, $(G_{st})_*(A)$ agrees with topological cyclic homology with a degree shift, and
Hesselholt and Madsen compute that the topological cyclic homology of $A$
vanishes in even degrees. So $(G_{st})_*(A)$ is uniquely $\ell$-divisible in odd positive degrees.
But we know that $(G_{st})_0(A) \cong \mathbb{Z}/p^n\mathbb{Z}$, and as a consequence of
Proposition~\ref{existence of multiplicative structure on stable g-thy}, we know that
$(G_{st})_i(A)$ is a $(G_{st})_0(A)$-module for all $i$. So the uniquely $\ell$-divisible summands
must be trivial. So $(G_{st})_i(A)$ is trivial for odd positive $i$.
\item
We have the inclusion of the $2$-skeleton $S^2\hookrightarrow BU(1)$ into the classifying space of
complex line bundles, and the map it induces of $g_{st}(A)$-cohomological Atiyah-Hirzebruch spectral sequences:
\[ \xymatrix{
H^p(BU(1); (g_{st}(A))^q)  \ar@{=>}[r]\ar[d] & g_{st}(A)^{p+q}(BU(1)) \ar[d] \\
H^p(S^2; (g_{st}(A))^q) \ar@{=>}[r] & g_{st}(A)^{p+q}(S^2). }\]
Since the homotopy groups of $g_{st}(A)$ and the cohomology groups of $BU(1)$ 
are both concentrated in even degrees, there is no room for differentials in the spectral sequence for $BU(1)$.
We conclude that the map of spectral sequences is simply projection on to the $p=0$ and $p=2$ columns, and that in particular, every element
in $H^2(S^2; (g_{st}(A))^0)\cong g_{st}(A)^0(\pt)$---in particular, the multiplicative unit element---is in the image of the projection from
$H^2(BU(1); (g_{st}(A))^0) \cong g_{st}(A)^2(BU(1))$. 
So $g_{st}(A)$ is complex oriented.
\end{itemize}
\end{proof}

\bibliography{/home/asalch/texmf/tex/salch}{}
\bibliographystyle{plain}
\end{document}